\documentclass[final,3p,times]{elsarticle}

\usepackage{amssymb,amscd, verbatim,color,fancyhdr, mathrsfs}
\usepackage{amsmath}
 \usepackage{amsthm}
 \usepackage[T1]{fontenc}
 
\usepackage[ansinew]{inputenc} 
\usepackage{mathrsfs}
\usepackage{euscript}
\usepackage{natbib}

\usepackage{graphicx}
\usepackage{turnstile}
\usepackage{tikz}
\usepackage{multicol}
\usetikzlibrary{patterns}

\newcommand{\vertiii}[1]{{\left\vert\kern-0.15ex\left\vert\kern-0.15ex\left\vert #1 
    \right\vert\kern-0.15ex\right\vert\kern-0.15ex\right\vert}}

\newtheorem{theorem}{Theorem}
\newtheorem{definition}{Definition}
\newtheorem{lemma}{Lemma}
\newdefinition{rmk}{Remark}
\newproof{pf}{Proof}

\begin{document}

\begin{frontmatter}

\title{Liouville-type theorems for sign-changing solutions to nonlocal elliptic inequalities and systems with variable-exponent nonlinearities}

\author[label]{Ahmad Z. Fino}
 \address[label]{Department of Mathematics, Faculty of Sciences, Lebanese University, P.O. Box 1352, Tripoli, Lebanon}
 \ead{ahmad.fino01@gmail.com; afino@ul.edu.lb}

\author[label1]{Mohamed Jleli}
\address[label1]{Department of Mathematics, College of Science, King Saud University, P.O. Box 2455, Riyadh, 11451, Saudia Arabia}
\ead{jleli@ksu.edu.sa}
 
\author[label1]{Bessem Samet}
\ead{bsamet@ksu.edu.sa}

\begin{abstract}
We consider the fractional elliptic inequality with variable-exponent nonlinearity 
$$
(-\Delta)^{\frac{\alpha}{2}} u+\lambda\, \Delta u \geq |u|^{p(x)}, \quad x\in\mathbb{R}^N,
$$
where $N\geq 1$, $\alpha\in (0,2)$, $\lambda\in\mathbb{R}$ is a  constant,   $p: \mathbb{R}^N\to (1,\infty)$ is a  measurable function, and $(-\Delta)^{\frac{\alpha}{2}}$ is the fractional Laplacian
operator of order $\frac{\alpha}{2}$. A Liouville-type theorem is established for the considered problem. Namely, we obtain sufficient conditions under which the only weak solution is the trivial one. Next, we extend our study to systems of fractional elliptic inequalities with variable-exponent nonlinearities. Besides the consideration of variable-exponent nonlinearities, the novelty of this work consists in investigating  sign-changing solutions to the considered problems. Namely, to the best of our knowledge, only nonexistence results of positive solutions to fractional elliptic problems were invetigated previously. Our approach is based on the nonlinear capacity method combined with  a pointwise estimate of the fractional Laplacian of some test functions, which was derived by Fujiwara (2018) (see also Dao and Reissig (2019)). Note that the standard nonlinear capacity method cannot be applied to the considered problems due to the change of sign of solutions. 
\end{abstract}

\end{frontmatter}

\section{Introduction}
\setcounter{equation}{0} 

The goal of this paper is to study the nonexistence of nontrivial sign-changing solutions  to a class of fractional elliptic inequalities and systems with variable-exponent nonlinearities, namely 
\begin{equation}\label{1}
(-\Delta)^{\frac{\alpha}{2}} u+\lambda\, \Delta u \geq |u|^{p(x)}, \quad x\in\mathbb{R}^N
\end{equation}
and
\begin{eqnarray}\label{2}
\left\{\begin{array}{lll}
(-\Delta)^{\frac{\alpha}{2}} u+\lambda\, \Delta u &\geq & |v|^{q(x)},\quad x\in\mathbb{R}^N,\\ \\
(-\Delta)^{\frac{\beta}{2}} v+\mu\, \Delta v &\geq &|u|^{p(x)},\quad x\in\mathbb{R}^N,
\end{array}
\right.
\end{eqnarray}
where $N\geq 1$, $\alpha,\beta\in (0,2)$, $\lambda,\mu\in\mathbb{R}$ are constants,   $p,q: \mathbb{R}^N\to (1,\infty)$ are measurable functions, and $(-\Delta)^{\frac{\kappa}{2}}$, $\kappa\in \{\alpha,\beta\}$, is the fractional Laplacian
operator of order $\frac{\kappa}{2}$. We mention below some motivations for studying problems of types \eqref{1} and \eqref{2}.

In \cite{GS}, Gidas and  Spruck  considered the corresponding equation 
to \eqref{1} with $\alpha=2$, $\lambda=0$,  $p(\cdot)\equiv p$ and $u\geq 0$, namely
\begin{eqnarray}\label{1-GS}
\left\{\begin{array}{lllll}
-\Delta u &= &  u^p &\mbox{in}& \mathbb{R}^N,\\
u &\geq & 0 &\mbox{in}& \mathbb{R}^N.
\end{array}
\right.
\end{eqnarray}
It was shown that, 
\begin{itemize}
\item[(a)] if $N\geq 3$ and $1<p<\frac{N+2}{N-2}$, then \eqref{1-GS} admits admits no positive classical solution;
\item[(b)] if $N\geq 3$ and $p\geq \frac{N+2}{N-2}$, then \eqref{1-GS} admits positive classical solutions.
\end{itemize}

Consider the corresponding system of equations to  \eqref{2} with $\alpha=\beta=2$, $\lambda=\mu=0$, $p(\cdot)\equiv p>0$, $q(\cdot)\equiv q>0$ and $u,v\geq 0$,  namely the Lane-Emden system 
\begin{eqnarray}\label{2-LES}
\left\{\begin{array}{lllll}
-\Delta u &= & v^q &\mbox{in}& \mathbb{R}^N,\\ 
-\Delta v &= &u^p&\mbox{in}& \mathbb{R}^N,\\ 
u &\geq &  0 &\mbox{in}& \mathbb{R}^N,\\ 
v &\geq &  0 &\mbox{in}& \mathbb{R}^N,
\end{array}
\right.
\end{eqnarray}
where $N\geq 3$.  The famous Lane-Emden conjecture
states that, if 
$$
\frac{1}{p+1}+\frac{1}{q+1}>1-\frac{2}{N},
$$
then \eqref{2-LES} admits no positive classical solution. This conjecture was proved only in the cases $N\in \{3,4\}$ 
(see \cite{PQS,SZ,Souplet}).

In the case $\alpha=2$, $\lambda=0$, $p(\cdot)\equiv p>1$ and $u\geq 0$,   \eqref{1} reduces to 
\begin{eqnarray}\label{1-NS}
\left\{\begin{array}{lllll}
-\Delta u &\geq  &  u^p &\mbox{in}& \mathbb{R}^N,\\
u &\geq & 0 &\mbox{in}& \mathbb{R}^N.
\end{array}
\right.
\end{eqnarray}
Ni and Serrin \cite{NS} investigated the radial case of \eqref{1-NS}. Namely, it was shown that, if $N\geq 3$ and $1<p\leq \frac{N}{N-2}$, then \eqref{1-NS} has no positive radial solution such that  
$\displaystyle\lim_{|x|\to \infty} u(|x|)=0$.

Mitidieri and Pohozaev \cite{MP} studied sign-changing solutions  to the differential inequality
\begin{equation}\label{1-MP}
-\Delta u \geq |u|^{p} \quad \mbox{in }\,\,  \mathbb{R}^N.
\end{equation}
In the case $N\geq 3$, it was shown that, if 
\begin{equation}\label{cdMP}
1<p\leq \frac{N}{N-2},
\end{equation}
then  \eqref{1-MP}  has no nontrivial weak solution. Note that \eqref{cdMP} is sharp, in the sense that, if $N\geq 3$ and $p>\frac{N}{N-2}$, then  \eqref{1-MP} admits positive classical solutions. Indeed, one can check easily that in this case, 
$$
u(x)=\epsilon \left(1+|x|^2\right)^{\frac{1}{1-p}},\quad x\in \mathbb{R}^N,
$$
is a positive solution  to \eqref{1-MP} for sufficiently small $\epsilon>0$.  

Mitidieri and Pohozaev \cite{MP} studied also the corresponding system to \eqref{1-MP}, namely
\begin{eqnarray}\label{2-MP}
\left\{\begin{array}{lll}
-\Delta u &\geq & |v|^{q},\quad x\in\mathbb{R}^N,\\ \\
-\Delta v &\geq &|u|^{p},\quad x\in\mathbb{R}^N.
\end{array}
\right.
\end{eqnarray}
It was shown that, if
\begin{equation}\label{sysMP}
N\leq \max\left\{\frac{2q(p+1)}{pq-1},\frac{2p(q+1)}{pq-1}\right\},
\end{equation}
then \eqref{2-MP} admits no nontrivial weak solution. Moreover, condition \eqref{sysMP} is sharp, in the sense that, if $N> \max\left\{\frac{2q(p+1)}{pq-1},\frac{2p(q+1)}{pq-1}\right\}$, then \eqref{2-MP} admits positive classical solutions ($u,v>0$). Indeed, it can be easily seen that in this case,
$$
(u(x),v(x))=\left(\epsilon \left(1+|x|^2\right)^{\frac{q+1}{1-pq}},\epsilon \left(1+|x|^2\right)^{\frac{p+1}{1-pq}}\right),\quad x\in \mathbb{R}^N,
$$
is a positive solution to \eqref{2-MP} for sufficiently small $\epsilon>0$.  

A large class of differential inequalities and systems generalizing \eqref{1-MP} and \eqref{2-MP} was systematically investigated  by Mitidieri and Pohozaev (see e.g. \cite{MP98,MP1,MP2,MP}), who developed the nonlinear capacity method. Next, this approach was used by many authors in the study of different types of problems (see e.g. \cite{BP,CAM,DYZ,Fi0,Fi,Sun}
).

Nonlocal operators  have been receiving increased attention in recent years
due to their usefulness in modeling complex systems with long-range interactions or memory effects, which cannot be described properly via standard differential operators. In particular, the fractional Laplacian $(-\Delta)^\frac{\alpha}{2}$, $0<\alpha<2$, has been used to describe anomalous diffusion \cite{MM}, turbulent flows \cite{GU}, stochastic dynamics \cite{BBC,Chen}, finance \cite{CO}, and many other phenomena. Due to the above facts, the study of mathematical problems involving the fractional Laplacian operator has attracted significant attention recently. In particular, many interesting results related to Liouville-type theorems for  nonlocal elliptic problems have been obtained.

To  overcome the difficulty caused by the nonlocal property of the fractional Laplacian operator, Caffarelli and Silvestre \cite{CS} introduced an extension method which consists of  localizing the fractional Laplacian by constructing a Dirichlet to Neumann operator of a degenerate elliptic equation.  Using the mentioned approach, Brandle et al. \cite{BR} established a nonexistence result for the fractional version of 
\eqref{1-GS}, namely
\begin{eqnarray}\label{F-1-GS}
\left\{\begin{array}{lllll}
(-\Delta)^{\frac{\alpha}{2}} u &=  &  u^p &\mbox{in}& \mathbb{R}^N,\\
u &\geq & 0 &\mbox{in}& \mathbb{R}^N.
\end{array}
\right.
\end{eqnarray}
It was shown that, if $1\leq \alpha<2$, $N\geq 2$ and $1<p<\frac{N+\alpha}{N-\alpha}$, then \eqref{F-1-GS} has no nontrivial bounded solution. 

In \cite{ZH}, Zhuo et al. investigated \eqref{F-1-GS} using an equivalent integral representation to \eqref{F-1-GS}.  They obtained the same result as in \cite{BR} but under weaker conditions. Namely, they proved that, if $0<\alpha<2$, $N\geq 2$ and $1<p<\frac{N+\alpha}{N-\alpha}$, then \eqref{F-1-GS} has no nontrivial locally bounded solution. 

In \cite{QX}, Quaas and Xia studied the fractional Lane-Emden system 
\begin{eqnarray}\label{F-2-LES}
\left\{\begin{array}{lllll}
(-\Delta)^{\frac{\alpha}{2}} u &= & v^q &\mbox{in}& \mathbb{R}^N,\\ 
(-\Delta)^{\frac{\alpha}{2}}v &= &u^p&\mbox{in}& \mathbb{R}^N,\\ 
u &\geq &  0 &\mbox{in}& \mathbb{R}^N,\\ 
v &\geq &  0 &\mbox{in}& \mathbb{R}^N,
\end{array}
\right.
\end{eqnarray}
where $0<\alpha<2$, $N>\alpha$ and $p,q>0$. Using the method of moving planes, it was shown that, if $pq>1$, 
$$
\beta_1,\beta_2\in \left[\frac{N-\alpha}{2},N-\alpha\right)\quad\mbox{and}\quad (\beta_1,\beta_2)\neq \left(\frac{N-\alpha}{2},\frac{N-\alpha}{2}\right),
$$
where $\beta_1=\frac{\alpha(q+1)}{pq-1}$ and $\beta_2=\frac{\alpha(p+1)}{pq-1}$, then for some $\sigma > 0$, there exists no positive solution to 
\eqref{F-2-LES} in $X_{\alpha,\sigma}(\mathbb{R}^N)$, where
$$
X_{\alpha,\sigma}(\mathbb{R}^N)=
\left\{\begin{array}{lll}
C^{\alpha+\sigma}(\mathbb{R}^N) &\mbox{if}& 0<\alpha<1,\\
C^{1,\alpha+\sigma-1}(\mathbb{R}^N) &\mbox{if}& 1\leq \alpha<2.
\end{array}
\right.
$$

In the case $\lambda=0$, $p(\cdot)\equiv p$ and $u\geq 0$, \eqref{1} reduces to 
\begin{eqnarray}\label{1-FFNS}
\left\{\begin{array}{lllll}
(-\Delta)^{\frac{\alpha}{2}} u &\geq  &  u^p &\mbox{in}& \mathbb{R}^N,\\
u &\geq & 0 &\mbox{in}& \mathbb{R}^N.
\end{array}
\right.
\end{eqnarray}
Using the extension method \cite{CS}, Wang and Xiao \cite{WX} proved the following results for \eqref{1-FFNS}: Let $p>1$, $0<\alpha<2$ and $N\geq 1$. Then
\begin{itemize}
\item[(a)] if $N\leq \alpha$, then \eqref{1-FFNS} has no nontrivial weak solution;
\item[(b)] if $N>\alpha$, then \eqref{1-FFNS} has no nontrivial weak solution when and only when $p\leq \frac{N}{N-\alpha}$.
\end{itemize}

In the special case $\lambda=\mu=0$, $p(\cdot)\equiv p$, $q(\cdot)\equiv q$ and $u,v\geq 0$, \eqref{2} reduces to 
\begin{eqnarray}\label{2-testf}
\left\{\begin{array}{lllll}
(-\Delta)^{\frac{\alpha}{2}} u &\geq & v^{q} &\mbox{in}& \mathbb{R}^N,\\ 
(-\Delta)^{\frac{\beta}{2}} v  &\geq & u^{p}&\mbox{in}& \mathbb{R}^N,\\
u &\geq &  0 &\mbox{in}& \mathbb{R}^N,\\ 
v &\geq &  0 &\mbox{in}& \mathbb{R}^N.
\end{array}
\right.
\end{eqnarray}
Using the nonlinear capacity method and Ju's inequality \cite{Ju}, Dahmani et al. \cite{DKK} proved that, if $0<\alpha,\beta<2$, $p,q>1$ and 
\begin{equation}\label{Nest}
N<\max\left\{\beta+\frac{\alpha}{q},\alpha+\frac{\beta}{p}\right\} \frac{pq}{pq-1},
\end{equation}
then \eqref{2-testf} has no nontrivial weak solution. Observe that in the case $p=q>1$, $0<\alpha=\beta<2$, $N>\alpha$  and $u=v\geq 0$, \eqref{Nest} reduces to $p<\frac{N}{N-\alpha}$, which is the sufficient condition for the nonexistence of nontrivial weak solution to \eqref{1-FFNS} obtained in the statement (b) (without the limit case $p=\frac{N}{N-\alpha}$).

In all the above mentioned results, the positivity of solutions is essential. In particular, the standard nonlinear capacity method used in \cite{DKK} cannot be applied to \eqref{1} and \eqref{2},  where solutions can change sign. Namely, the main difficulty consists in constructing a function $\theta\in C_0^\infty(\mathbb{R}^N)$, $\theta\geq 0$,  so that 
$$
\left|(-\Delta)^{\frac{\kappa}{2}}\theta^\ell(x)\right|\leq C \theta^{\ell-1}(x) \left|(-\Delta)^{\frac{\kappa}{2}}\theta(x)\right|,\quad x\in \mathbb{R}^N,
$$
where $\kappa\in (0,2)$,  $\ell>1$ and $C>0$ is a constant (independent on $x$).

The originality of this work  resides in considering sign-changing solutions 
 to fractional elliptic inequalities and systems, as well as variable exponents. As in \cite{DKK}, we use the nonlinear capacity method, but with a different choice of the test function, allowing us to  treat the sign-changing solutions case. This choice is motivated by the recent work of 
Dao and Reissig \cite{DR}, where a blow-up result  for semi-linear structurally damped $\sigma$-evolution equations  was derived. 

Before stating our main results, we recall some basic notions related to the fractional Laplacian operator and Lebesgue spaces with  variable exponents, and define weak solutions to \eqref{1} and \eqref{2}. For more details about these notions, we refer to \cite{DHHMS,DHHMS2,Kwanicki,Silvestre} and the references therein.

Let $s\in (0,1)$. The fractional Laplacian operator $(-\Delta)^s$ is defined as 
\begin{equation}\label{fracL}
(-\Delta)^s f(x)=C_{N,s}\, P.V. \, \int_{\mathbb{R}^N}\frac{f(x)- f(y)}{|x-y|^{N+2s}}\, dy,\quad x\in \mathbb{R}^N,
\end{equation}
where $f$ belongs to a suitable set of functions,  $P.V.$ stands for Cauchy's principal value,  and $C_{N,s}>0$ is a  normalization constant that depends only on $N$ and $s$.

Let $p: \mathbb{R}^N\to (1,\infty)$ be a measurable function such that 
\begin{equation}\label{asm1}
1<p^-:=\mbox{ess} \inf_{x\in \mathbb{R}^N}p(x)\leq p(x)\leq p^+:=\mbox{ess} \sup_{x\in \mathbb{R}^N}p(x)<\infty,\quad x\in \mathbb{R}^N\, a.e. 
\end{equation}
The variable exponent Lebesgue space  $L^{p(\cdot)}(\mathbb{R}^N)$ is defined by 
$$
L^{p(\cdot)}(\mathbb{R}^N)=\left\{f:\mathbb{R}^N\to \mathbb{R}: f\mbox{ is measurable},\, \varrho_{p(\cdotp)}(\lambda f)<\infty \mbox{ for some }\lambda>0 \right\},
$$
where the modular $\varrho_{p(\cdotp)}$ is defined by
$$
\varrho_{p(\cdotp)}(f)=\int_{\mathbb{R}^N}| f(x)|^{p(x)}\,dx,
$$
with a Luxemburg-type norm
$$
\|f\|_{p(\cdotp)}=\inf\left\{\lambda>0:\,\varrho_{p(\cdotp)}\left(\frac{f}{\lambda}\right)\leq 1\right\},\quad f\in L^{p(\cdot)}(\mathbb{R}^N).
$$
Equipped with this norm, $L^{p(\cdotp)}(\mathbb{R}^N)$ is a Banach space. Moreover, one has
\begin{equation}\label{esve}
\min\left\{\varrho_{p(\cdotp)}(f)^{\frac{1}{p^-}},\varrho_{p(\cdotp)}(f)^{\frac{1}{p^+}}\right\}\leq \|f\|_{p(\cdotp)}\leq \max\left\{\varrho_{p(\cdotp)}(f)^{\frac{1}{p^-}},\varrho_{p(\cdotp)}(f)^{\frac{1}{p^+}}\right\},\quad f\in L^{p(\cdot)}(\mathbb{R}^N).
\end{equation}

Problem \eqref{1} is investigated under the following assumptions: $N\geq 1$, $0<\alpha<2$, $\lambda\in \mathbb{R}$, and $p: \mathbb{R}^N\to (1,\infty)$ is  a measurable function satisfying \eqref{asm1}. 

\begin{definition}[Weak solution for \eqref{1}]\label{weakequation}  
We say that $u  \in L^{2}(\mathbb{R}^N)\cap L^{2p(\cdot)}(\mathbb{R}^N)$ is a weak solution to \eqref{1} if
$$
\int_{\mathbb{R}^N}|u(x)|^{p(x)}\varphi(x)\,dx\leq \int_{\mathbb{R}^N}u(x)(-\Delta)^{\frac{\alpha}{2}}\varphi(x)\,dx+\lambda \int_{\mathbb{R}^N}u(x)\Delta\varphi(x)\,dx,
$$
for all $\varphi\in H^2(\mathbb{R}^N)$,  $\varphi\geq 0$. 
\end{definition}

Problem \eqref{2} is investigated under the following assumptions: $N\geq 1$, $0<\alpha,\beta<2$, $\lambda,\mu\in \mathbb{R}$, and $p,q: \mathbb{R}^N\to (1,\infty)$   are measurable functions satisfying respectively \eqref{asm1} and
$$
1<q^-:=\mbox{ess} \inf_{x\in \mathbb{R}^N}q(x)\leq q(x)\leq q^+:=\mbox{ess} \sup_{x\in \mathbb{R}^N}q(x)<\infty,\quad x\in \mathbb{R}^N\, a.e. 
$$

\begin{definition}[Weak solution for \eqref{2}]
We say that 
$$
(u,v) \in (L^{2}(\mathbb{R}^N)\cap L^{2p(\cdotp)}(\mathbb{R}^N))\times (L^{2}(\mathbb{R}^N)\cap L^{2q(\cdotp)}(\mathbb{R}^N))
$$
is a weak solution to  \eqref{2} if
\begin{equation}\label{weaksystem1}  
\int_{\mathbb{R}^N}|v(x)|^{q(x)}\varphi(x)\,dx\leq \int_{\mathbb{R}^N}u(x)(-\Delta)^{\frac{\alpha}{2}}\varphi(x)\,dx+\lambda \int_{\mathbb{R}^N}u(x)\Delta\varphi(x)\,dx
\end{equation}
and
\begin{equation}\label{weaksystem2}  
\int_{\mathbb{R}^N}|u(x)|^{p(x)}\varphi(x)\,dx\leq \int_{\mathbb{R}^N}v(x)(-\Delta)^{\frac{\beta}{2}}\varphi(x)\,dx+\mu \int_{\mathbb{R}^N}v(x)\Delta\varphi(x)\,dx,
\end{equation}
for all  $\varphi\in H^2(\mathbb{R}^N)$, $\varphi\geq 0$. 
\end{definition}

Now, we are ready to state the main results of this paper.

\begin{theorem}\label{theorem1}
If 
\begin{equation}\label{cdnonexist1}
1<p^-\leq p^+< p^*(N),
\end{equation}
where
$$
p^{*}(N)=\left\{\begin{array}{lll}
\infty &\mbox{if}& N\leq\alpha,\\
\frac{N}{N-\alpha} &\mbox{if}& N>\alpha,
\end{array}
\right.
$$
then the only weak solution to \eqref{1} is the trivial one.
\end{theorem}

\begin{rmk}
When $N>\alpha$, it is still an open question whether or not \eqref{1} admits nontrivial sign-changing weak solutions in the following cases:
\begin{itemize}
\item[(a)] $1<p^-\leq p^*(N)\leq p^+$;
\item[(b)] $1<p^*(N)<p^-$.
\end{itemize}
Note that in the case $\lambda=0$ and $p(\cdot)\equiv p$ (i.e. $p^-=p^+=p$), (a) reduces to 
\begin{itemize}
\item[(a')] $p=p^*(N)$,
\end{itemize}
while (b) reduces to 
\begin{itemize}
\item[(b')] $p>p^*(N)$.
\end{itemize} 
From \cite[Theorem 1.1 (ii)]{WX}, in the case (a'), the only nonnegative weak solution to this special case of \eqref{1} is the trivial one; while in the case (b'), positive strong solutions exist. 
\end{rmk}

\begin{rmk}
(i) The exponent $p^*(N)$  depends only on the dimension and the lower order power of $(-\Delta)$, i.e. on $\alpha$.\\
(ii) In the case $\alpha=2$, $p^*(N)$ is the critical exponent for problem \eqref{1-MP}.
\end{rmk}

\begin{theorem}\label{theorem2}
If 
\begin{equation}\label{condsyst}
N< \frac{p^+q^+}{p^+q^+-1} \max\left\{\beta+\frac{\alpha}{q^+},\alpha+\frac{\beta}{p^+}\right\},
\end{equation}
then the only weak solution to \eqref{2} is the trivial one, i.e. $(u,v)\equiv (0,0)$.
\end{theorem}

\begin{rmk}
(i) In the case $\alpha=\beta=2$, $p(\cdot)\equiv p$ and $q(\cdot)\equiv q$, \eqref{condsyst} reduces to \eqref{sysMP} (with strict inequality), which is the obtained condition in \cite{MP}, under which \eqref{2-MP} admits no nontrivial weak solution.\\ (ii) In the case  $p(\cdot)\equiv p$ and $q(\cdot)\equiv q$, \eqref{condsyst} reduces to \eqref{Nest}, which is the obtained condition in \cite{DKK}, under which \eqref{2-testf} has no positive weak solution.  
\end{rmk}

The proofs of our main results are given in the next section.

\section{Proofs}\label{sec2}

In this section, we give the proofs of Theorems \ref{theorem1} and \ref{theorem2}. We shall use the nonlinear capacity method combined with the following pointwise estimate (see Fujiwara \cite{Fuj} and Dao and Reissig \cite{DR}).

\begin{lemma}\label{lemma1}
Let 
$$
\langle x\rangle:=(1+|x|^2)^{\frac{1}{2}},\quad x\in \mathbb{R}^N.
$$
Let $s \in (0,1]$ and $\theta: \mathbb{R}^N\to (0,\infty)$ be the function defined by  
\begin{equation}\label{testfOK}
\theta(x)=\langle x\rangle^{-\rho},\quad x\in \mathbb{R}^N,
\end{equation}
where $N<\rho\leq N+2s$. Then $\theta\in L^1(\mathbb{R}^N)\cap H^2(\mathbb{R}^N)$, and the following estimate holds:
\begin{equation}\label{3}
\left|(-\Delta)^s\theta(x)\right|\leq C  \theta(x), \quad x\in\mathbb{R}^N,
\end{equation}
where $C>0$ is a constant (independent of $x$). 
\end{lemma}

\begin{proof}[Proof of Theorem \ref{theorem1}]
Let  $u  \in L^{2}(\mathbb{R}^N)\cap L^{2p(\cdot)}(\mathbb{R}^N)$ be a weak solution to \eqref{1}.  By Definition \ref{weakequation}, for all $\varphi\in H^2(\mathbb{R}^N)$, $\varphi\geq 0$, one has 
\begin{equation}\label{gest}
\int_{\mathbb{R}^N}|u(x)|^{p(x)}\varphi(x)\,dx\leq \int_{\mathbb{R}^N}|u(x)| \left|(-\Delta)^{\frac{\alpha}{2}}\varphi(x)\right|\,dx+|\lambda| \int_{\mathbb{R}^N}|u(x)|\left|\Delta\varphi(x)\right|\,dx.
\end{equation}
On the other hand, for all  $0<\epsilon<1$ and $x\in \mathbb{R}^N$ a.e, writing
$$
|u(x)| \left|(-\Delta)^{\frac{\alpha}{2}}\varphi(x)\right|=\left[\left(\epsilon p(x)\varphi(x)\right)^{\frac{1}{p(x)}} |u(x)|\right]\left[\left(\epsilon p(x)\varphi(x)\right)^{\frac{-1}{p(x)}} \left|(-\Delta)^{\frac{\alpha}{2}}\varphi(x)\right|\right]
$$
and using Young's inequality, it holds that  
$$
|u(x)| \left|(-\Delta)^{\frac{\alpha}{2}}\varphi(x)\right|\leq \epsilon |u(x)|^{p(x)}\varphi(x)+\epsilon^{\frac{-1}{p(x)-1}}\left(\frac{p(x)-1}{p(x)}\right) p(x)^{\frac{-1}{p(x)-1}}\varphi(x)^{\frac{-1}{p(x)-1}}\left|(-\Delta)^{\frac{\alpha}{2}}\varphi(x)\right|^{\frac{p(x)}{p(x)-1}}.
$$
Next, using \eqref{asm1}, one obtains
$$
|u(x)| \left|(-\Delta)^{\frac{\alpha}{2}}\varphi(x)\right|\leq \epsilon  |u(x)|^{p(x)}\varphi(x)+\epsilon^{\frac{-1}{p^--1}} (p^+)^{\frac{-1}{p^+-1}}\varphi(x)^{\frac{-1}{p(x)-1}}\left|(-\Delta)^{\frac{\alpha}{2}}\varphi(x)\right|^{\frac{p(x)}{p(x)-1}},
$$
which yields
\begin{equation}\label{firstes}
\int_{\mathbb{R}^N}|u(x)| \left|(-\Delta)^{\frac{\alpha}{2}}\varphi(x)\right|\,dx\leq \epsilon \int_{\mathbb{R}^N} |u(x)|^{p(x)}\varphi(x)\,dx+C \int_{\mathbb{R}^N} \varphi(x)^{\frac{-1}{p(x)-1}}\left|(-\Delta)^{\frac{\alpha}{2}}\varphi(x)\right|^{\frac{p(x)}{p(x)-1}}\,dx.
\end{equation}
Throughout, $C$ denotes a positive constant, whose value may change from line to line. Similarly, one obtains
\begin{equation}\label{secondes}
\int_{\mathbb{R}^N}|u(x)|\left|\Delta\varphi(x)\right|\,dx\leq \epsilon \int_{\mathbb{R}^N} |u(x)|^{p(x)}\varphi(x)\,dx+C \int_{\mathbb{R}^N} \varphi(x)^{\frac{-1}{p(x)-1}}\left|\Delta \varphi(x)\right|^{\frac{p(x)}{p(x)-1}}\,dx.
\end{equation}
Next, taking 
$$
0<\varepsilon<\frac{1}{1+|\lambda|},
$$
it follows from \eqref{gest}, \eqref{firstes} and \eqref{secondes} that 
\begin{equation}\label{es3}
\int_{\mathbb{R}^N}|u(x)|^{p(x)}\varphi(x)\,dx\leq C \left(I_1(\varphi)+|\lambda| I_2(\varphi)\right),
\end{equation}
where
$$
I_1(\varphi):=\int_{\mathbb{R}^N} \varphi(x)^{\frac{-1}{p(x)-1}}\left|(-\Delta)^{\frac{\alpha}{2}}\varphi(x)\right|^{\frac{p(x)}{p(x)-1}}\,dx\quad\mbox{and}\quad 
I_2(\varphi):=\int_{\mathbb{R}^N} \varphi(x)^{\frac{-1}{p(x)-1}}\left|\Delta\varphi(x)\right|^{\frac{p(x)}{p(x)-1}}\,dx.
$$
Now, for $R>1$, we take
$$
\varphi(x)=\theta\left(\frac{x}{R}\right),\quad x\in \mathbb{R}^N,
$$
where $\theta$ is the function defined by \eqref{testfOK} with $\rho=N+\alpha$.  Using \eqref{fracL}, the change of variable $x=Ry$, and \cite[Lemma 2.4]{DR}, one obtains
$$
I_1(\varphi)=R^N \int_{\mathbb{R}^N} R^{\frac{-\alpha p(Ry)}{p(Ry)-1}}\theta(y)^{\frac{-1}{p(Ry)-1}} \left|(-\Delta)^{\frac{\alpha}{2}}\theta(y)\right|^{\frac{p(Ry)}{p(Ry)-1}}\,dy\quad\mbox{and}\quad 
I_2(\varphi)=R^{N} \int_{\mathbb{R}^N} R^{\frac{-2 p(Ry)}{p(Ry)-1}}\theta(y)^{\frac{-1}{p(Ry)-1}} \left|\Delta\theta(y)\right|^{\frac{p(Ry)}{p(Ry)-1}}\,dy.
$$
On the other hand, by Lemma \ref{lemma1}, one has 
$$
\left|(-\Delta)^{\frac{\kappa}{2}}\theta(y)\right|\leq C \theta(y),\quad y\in \mathbb{R}^N, 
$$
where $\kappa\in \{\alpha,2\}$. Hence, one deduces that 
\begin{eqnarray}\label{sim1}
\nonumber I_1(\varphi) & \leq & C R^{N} \int_{\mathbb{R}^N}R^{\frac{-\alpha p(Ry)}{p(Ry)-1}}\theta(y)\,dy\\
\nonumber &\leq &  C R^{N} R^{\frac{-\alpha p^+}{p^+-1}} \int_{\mathbb{R}^N}\theta(y)\,dy\\
&=& C R^{N-\frac{\alpha p^+}{p^+-1}}.
\end{eqnarray}
Similarly, one has
\begin{equation}\label{sim2}
I_2(\varphi) \leq C R^{N-\frac{2 p^+}{p^+-1}}.
\end{equation}
Therefore, using \eqref{es3}, one deduces that 
$$
\int_{\mathbb{R}^N}|u(x)|^{p(x)}\varphi_R(x)\,dx\leq C 
\left(R^{N-\frac{\alpha p^+}{p^+-1}}+|\lambda| R^{N-\frac{2 p^+}{p^+-1}}\right),\quad R>1.
$$
Finally, passing to the infimum limit as $R\to \infty$ in the above inequality, using Fatou's Lemma and \eqref{cdnonexist1}, one obtains
$$
\varrho_{p(\cdotp)}(u)=\int_{\mathbb{R}^N}|u(x)|^{p(x)}\,dx=0,
$$
which implies by \eqref{esve} that $\|u\|_{p(\cdotp)}=0$, i.e. $u$ is the trivial solution. This completes the proof of Theorem \ref{theorem1}.
\end{proof}

\begin{proof}[Proof of Theorem \ref{theorem2}]
Let 
$$
(u,v) \in (L^{2}(\mathbb{R}^N)\cap L^{2p(\cdotp)}(\mathbb{R}^N))\times (L^{2}(\mathbb{R}^N)\cap L^{2q(\cdotp)}(\mathbb{R}^N))
$$
be a weak solution to \eqref{2}. By \eqref{weaksystem1} and \eqref{weaksystem2},   for all $\varphi\in H^2(\mathbb{R}^N)$, $\varphi\geq 0$, one has 
\begin{equation}\label{weaksystem1P}  
Y\leq \int_{\mathbb{R}^N}|u(x)| \left|(-\Delta)^{\frac{\alpha}{2}}\varphi(x)\right|\,dx+|\lambda| \int_{\mathbb{R}^N}|u(x)||\Delta\varphi(x)|\,dx
\end{equation}
and
\begin{equation}\label{weaksystem2P}  
X \leq \int_{\mathbb{R}^N}|v(x)| \left|(-\Delta)^{\frac{\beta}{2}}\varphi(x)\right|\,dx+|\mu| \int_{\mathbb{R}^N}|v(x)||\Delta\varphi(x)|\,dx,
\end{equation}
where
$$
X:=\int_{\mathbb{R}^N}|u(x)|^{p(x)}\varphi(x)\,dx\quad\mbox{and}\quad Y:=\int_{\mathbb{R}^N}|v(x)|^{q(x)}\varphi(x)\,dx.
$$
On the other hand, using H\"older's inequality, one obtains
\begin{eqnarray*}
\int_{\mathbb{R}^N}|u(x)| \left|(-\Delta)^{\frac{\alpha}{2}}\varphi(x)\right|\,dx&=&\int_{\{x\in \mathbb{R}^N:\, |u(x)|<1\}}|u(x)|\left|(-\Delta)^{\frac{\alpha}{2}}\varphi(x)\right|\,dx+ \int_{\{x\in \mathbb{R}^N:\, |u(x)|\geq 1\}}|u(x)|\left|(-\Delta)^{\frac{\alpha}{2}}\varphi(x)\right|\,dx\\
&\leq & \left(\int_{\{x\in \mathbb{R}^N:\, |u(x)|<1\}}|u(x)|^{p^+}\varphi(x)\,dx\right)^{\frac{1}{p^+}}\left(\int_{\{x\in \mathbb{R}^N:\, |u(x)|<1\}}\varphi(x)^{-\frac{1}{p^+-1}}\left|(-\Delta)^{\frac{\alpha}{2}}\varphi(x)\right|^{\frac{p^+}{p^+-1}}\,dx\right)^{\frac{p^+-1}{p^+}}\\
&&+ \left(\int_{\{x\in \mathbb{R}^N:\, |u(x)|\geq 1\}}|u(x)|^{p^-}\varphi(x)\,dx\right)^{\frac{1}{p^-}}\left(\int_{\{x\in \mathbb{R}^N:\, |u(x)|\geq 1\}}\varphi(x)^{-\frac{1}{p^--1}}\left|(-\Delta)^{\frac{\alpha}{2}}\varphi(x)\right|^{\frac{p^-}{p^--1}}\,dx\right)^{\frac{p^--1}{p^-}}\\
&\leq & X^{\frac{1}{p^+}} \left(\int_{\mathbb{R}^N}\varphi(x)^{-\frac{1}{p^+-1}}\left|(-\Delta)^{\frac{\alpha}{2}}\varphi(x)\right|^{\frac{p^+}{p^+-1}}\,dx\right)^{\frac{p^+-1}{p^+}}+X^{\frac{1}{p^-}}
\left(\int_{\mathbb{R}^N}\varphi(x)^{-\frac{1}{p^--1}}\left|(-\Delta)^{\frac{\alpha}{2}}\varphi(x)\right|^{\frac{p^-}{p^--1}}\,dx\right)^{\frac{p^--1}{p^-}},
\end{eqnarray*}
i.e.
\begin{equation}\label{aya1}
\int_{\mathbb{R}^N}|u(x)| \left|(-\Delta)^{\frac{\alpha}{2}}\varphi(x)\right|\,dx \leq [F(\alpha,p^+,\varphi)] ^{\frac{p^+-1}{p^+}} X^{\frac{1}{p^+}} + [F(\alpha,p^-,\varphi)] ^{\frac{p^--1}{p^-}} X^{\frac{1}{p^-}},
\end{equation}
where
$$
F(\kappa,r,\psi):=\int_{\mathbb{R}^N}\psi(x)^{-\frac{1}{r-1}}\left|(-\Delta)^{\frac{\kappa}{2}}\psi(x)\right|^{\frac{r}{r-1}}\,dx,
$$
for all $\kappa\in (0,2]$, $r>1$ and $\psi\in H^2(\mathbb{R}^N)$, $\psi\geq 0$.  Similarly, one obtains
\begin{equation}\label{aya2}
\int_{\mathbb{R}^N}|u(x)| |\Delta\varphi(x)|\,dx \leq [F(2,p^+,\varphi)] ^{\frac{p^+-1}{p^+}} X^{\frac{1}{p^+}} + [F(2,p^-,\varphi)] ^{\frac{p^--1}{p^-}} X^{\frac{1}{p^-}},
\end{equation}

\begin{equation}\label{aya3}
\int_{\mathbb{R}^N}|v(x)| \left|(-\Delta)^{\frac{\beta}{2}}\varphi(x)\right|\,dx \leq [F(\beta,q^+,\varphi)] ^{\frac{q^+-1}{q^+}} Y^{\frac{1}{q^+}} + [F(\beta,q^-,\varphi)] ^{\frac{q^--1}{q^-}} Y^{\frac{1}{q^-}},
\end{equation}
and
\begin{equation}\label{aya4}
\int_{\mathbb{R}^N}|v(x)| |\Delta\varphi(x)|\,dx \leq [F(2,q^+,\varphi)] ^{\frac{q^+-1}{q^+}} Y^{\frac{1}{q^+}} + [F(2,q^-,\varphi)] ^{\frac{q^--1}{q^-}} Y^{\frac{1}{p^-}}.
\end{equation}
Next, it follows from \eqref{weaksystem1P}--\eqref{aya4} that 
\begin{eqnarray}\label{NestI}
\left\{\begin{array}{lll}
X &\leq & A(\varphi) Y^{\frac{1}{q^+}}+B(\varphi)Y^{\frac{1}{q^-}},\\
Y &\leq & \overline{A(\varphi)} X^{\frac{1}{p^+}}+\overline{B(\varphi)}X^{\frac{1}{p^-}},
\end{array}
\right.
\end{eqnarray}
where 
$$
A(\varphi):=[F(\beta,q^+,\varphi)] ^{\frac{q^+-1}{q^+}}+|\mu|[F(2,q^+,\varphi)] ^{\frac{q^+-1}{q^+}},\quad B(\varphi):=[F(\beta,q^-,\varphi)] ^{\frac{q^--1}{q^-}}+|\mu|[F(2,q^-,\varphi)] ^{\frac{q^--1}{q^-}}
$$
and
$$
\overline{A(\varphi)}:=[F(\alpha,p^+,\varphi)] ^{\frac{p^+-1}{p^+}}+|\lambda|[F(2,p^+,\varphi)] ^{\frac{p^+-1}{p^+}},\quad \overline{B(\varphi)}:=[F(\alpha,p^-,\varphi)] ^{\frac{p^--1}{p^-}}+|\lambda|[F(2,p^-,\varphi)] ^{\frac{p^--1}{p^-}}.
$$
Thanks to the inequality
$$
(a+b)^m \leq 2^m (a^m+b^m),\quad a,b\geq 0,\, m>0,
$$
from \eqref{NestI}, one deduces that 
$$
Y^r\leq C \left(\overline{A(\varphi)}^rX^{\frac{r}{p^+}}+\overline{B(\varphi)}^rX^{\frac{r}{p^-}}\right),\quad r^{-1}\in \{q^+,q^-\},
$$
which yields
\begin{equation}\label{esXGG}
X \leq C\left( A(\varphi) \overline{A(\varphi)}^{\frac{1}{q^+}}X^{\frac{1}{q^+p^+}}+A(\varphi)\overline{B(\varphi)}^{\frac{1}{q^+}}X^{\frac{1}{q^+p^-}}+
B(\varphi)\overline{A(\varphi)}^{\frac{1}{q^-}} X^{\frac{1}{q^-p^+}}+B(\varphi)\overline{B(\varphi)}^{\frac{1}{q^-}}X^{\frac{1}{q^-p^-}}\right).
\end{equation}
On the other hand, using $\varepsilon$-Young inequality with $0<\varepsilon\ll 1$, one obtains
\begin{eqnarray}
A(\varphi) \overline{A(\varphi)}^{\frac{1}{q^+}}X^{\frac{1}{q^+p^+}}\leq \varepsilon X +C \left[A(\varphi) \overline{A(\varphi)}^{\frac{1}{q^+}}\right]^{\frac{q^+p^+}{q^+p^+-1}},\\
A(\varphi)\overline{B(\varphi)}^{\frac{1}{q^+}}X^{\frac{1}{q^+p^-}}\leq  \varepsilon X +C \left[A(\varphi) \overline{B(\varphi)}^{\frac{1}{q^+}}\right]^{\frac{q^+p^-}{q^+p^--1}},\\
B(\varphi)\overline{A(\varphi)}^{\frac{1}{q^-}} X^{\frac{1}{q^-p^+}}\leq \varepsilon X +C \left[B(\varphi) \overline{A(\varphi)}^{\frac{1}{q^-}}\right]^{\frac{q^-p^+}{q^-p^+-1}},\\
B(\varphi)\overline{B(\varphi)}^{\frac{1}{q^-}} X^{\frac{1}{q^-p^-}}\leq \varepsilon X +C \left[B(\varphi) \overline{B(\varphi)}^{\frac{1}{q^-}}\right]^{\frac{q^-p^-}{q^-p^--1}}.\label{IneqV}
\end{eqnarray}
Therefore, it follows from \eqref{esXGG}--\eqref{IneqV} that 
\begin{equation}\label{XESG}
X\leq C\left(\left[A(\varphi) \overline{A(\varphi)}^{\frac{1}{q^+}}\right]^{\frac{q^+p^+}{q^+p^+-1}}+\left[A(\varphi) \overline{B(\varphi)}^{\frac{1}{q^+}}\right]^{\frac{q^+p^-}{q^+p^--1}}+\left[B(\varphi) \overline{A(\varphi)}^{\frac{1}{q^-}}\right]^{\frac{q^-p^+}{q^-p^+-1}}+\left[B(\varphi) \overline{B(\varphi)}^{\frac{1}{q^-}}\right]^{\frac{q^-p^-}{q^-p^--1}}\right).
\end{equation}
Similarly, one obtains
\begin{equation}\label{YESG}
Y\leq C\left(\left[\overline{A(\varphi)} A(\varphi)^{\frac{1}{p^+}}\right]^{\frac{q^+p^+}{q^+p^+-1}}+\left[\overline{A(\varphi)} B(\varphi)^{\frac{1}{p^+}}\right]^{\frac{p^+q^-}{p^+q^--1}}+\left[\overline{B(\varphi)}A(\varphi)^{\frac{1}{p^-}}\right]^{\frac{p^-q^+}{p^-q^+-1}}+\left[\overline{B(\varphi)}B(\varphi)^{\frac{1}{p^-}}\right]^{\frac{q^-p^-}{q^-p^--1}}\right).
\end{equation}
For $R>1$, we take
$$
\varphi(x)=\theta\left(\frac{x}{R}\right),\quad x\in \mathbb{R}^N,
$$
where $\theta$ is the function defined by \eqref{testfOK} with $\rho=N+\min\{\alpha,\beta\}$.   Next, we have to estimate the terms $A(\varphi)$, $B(\varphi)$, $\overline{A(\varphi)}$ and $\overline{B(\varphi)}$. Similarly to \eqref{sim1} and \eqref{sim2}, one has
$$
F(\beta,q^+,\varphi)\leq C R^{N-\frac{\beta q^+}{q^+-1}}\quad\mbox{and}\quad 
F(2,q^+,\varphi)\leq C R^{N-\frac{2 q^+}{q^+-1}}.
$$
Using the above estimates, one deduces that 
\begin{equation}\label{esAphi}
A(\varphi)\leq C R^{N\left(\frac{q^+-1}{q^+}\right)-\beta}.
\end{equation}
Similarly, one obtains
\begin{equation}\label{esBphi}
B(\varphi)\leq C R^{N\left(\frac{q^--1}{q^-}\right)-\beta},
\end{equation}
\begin{equation}\label{esAphibar}
\overline{A(\varphi)}\leq C R^{N\left(\frac{p^+-1}{p^+}\right)-\alpha}\end{equation}
and
\begin{equation}\label{esBphibar}
\overline{B(\varphi)}\leq C R^{N\left(\frac{p^--1}{p^-}\right)-\alpha}.
\end{equation}
Therefore, using \eqref{XESG} and the estimates \eqref{esAphi}--\eqref{esBphibar}, one deduces that 
\begin{equation}\label{FestCompX}
X\leq C \sum_{i=1}^4 R^{\sigma_i},
\end{equation}
where
$$
\sigma_1:= \left(\frac{q^+p^+}{q^+p^+-1}\right)\left[N\left(\frac{p^+q^+-1}{p^+q^+}\right)-\beta-\frac{\alpha}{q^+}\right]:=\left(\frac{q^+p^+}{q^+p^+-1}\right) \overline{\sigma_1},
$$

$$
\sigma_2 := \left(\frac{q^+p^-}{q^+p^--1}\right) \left[N\left(\frac{p^-q^+-1}{p^-q^+}\right)-\beta-\frac{\alpha}{q^+}\right]:=\left(\frac{q^+p^-}{q^+p^--1}\right) \overline{\sigma_2},
$$
$$
\sigma_3:= \left(\frac{q^-p^+}{q^-p^+-1}\right) \left[N\left(\frac{p^+q^--1}{p^+q^-}\right)-\beta-\frac{\alpha}{q^-}\right]:=\left(\frac{q^-p^+}{q^-p^+-1}\right) \overline{\sigma_3}
$$
and
$$
\sigma_4:= \left(\frac{q^-p^-}{q^-p^--1}\right) \left[N\left(\frac{p^-q^--1}{p^-q^-}\right)-\beta-\frac{\alpha}{q^-}\right]:=\left(\frac{q^-p^-}{q^-p^--1}\right) \overline{\sigma_4}.
$$
One observes easily that 
\begin{equation}\label{sigma1}
\overline{\sigma_1}=\max\{\overline{\sigma_i}:\, i=1,2,3,4\}.
\end{equation}
Similarly, using \eqref{YESG} and the estimates \eqref{esAphi}--\eqref{esBphibar}, one deduces that 
\begin{equation}\label{secondCompY}
Y\leq C \sum_{i=1}^4 R^{\nu_i},
\end{equation}
where
$$
\nu_1:= \left(\frac{q^+p^+}{q^+p^+-1}\right)\left[N\left(\frac{p^+q^+-1}{p^+q^+}\right)-\alpha-\frac{\beta}{p^+}\right]:=\left(\frac{q^+p^+}{q^+p^+-1}\right)\overline{\nu_1}, 
$$
$$
\nu_2:= \left(\frac{p^+q^-}{p^+q^--1}\right) \left[N\left(\frac{q^-p^+-1}{q^-p^+}\right)-\alpha-\frac{\beta}{p^+}\right]:=\left(\frac{p^+q^-}{p^+q^--1}\right) \overline{\nu_2},
$$
$$
\nu_3:= \left(\frac{p^-q^+}{p^-q^+-1}\right) \left[N\left(\frac{q^+p^--1}{q^+p^-}\right)-\alpha-\frac{\beta}{p^-}\right]:=\left(\frac{p^-q^+}{p^-q^+-1}\right)\overline{\nu_3}
$$
and
$$
\nu_4:= \left(\frac{p^-q^-}{p^-q^--1}\right) \left[N\left(\frac{q^-p^--1}{q^-p^-}\right)-\alpha-\frac{\beta}{p^-}\right]:=\left(\frac{p^-q^-}{p^-q^--1}\right) \overline{\nu_4}.
$$
Moreover, one has
\begin{equation}\label{nu1}
\overline{\nu_1}=\max\{\overline{\nu_i}:\, i=1,2,3,4\}.
\end{equation}
Note that condition \eqref{condsyst} is equivalent to 
$$
\overline{\sigma_1}<0\quad\mbox{or}\quad\overline{\nu_1}<0.
$$ 
If $ \overline{\sigma_1}<0$, passing to  the infimum limit as $R\to \infty$ in  
\eqref{FestCompX}, using \eqref{sigma1} and Fatou's Lemma, one deduces that 
$$
\varrho_{p(\cdotp)}(u)=\int_{\mathbb{R}^N}|u(x)|^{p(x)}\,dx=0,
$$
which implies by \eqref{esve} that $\|u\|_{p(\cdotp)}=0$, i.e. $u\equiv 0$. 
Then, by \eqref{weaksystem1P}, one obtains $Y=0$, which yields  $v\equiv 0$. Therefore, $(u,v)\equiv (0,0)$. Similarly, if $ \overline{\nu_1}<0$, passing to  the infimum limit as $R\to \infty$ in  \eqref{secondCompY},
using \eqref{nu1} and Fatou's Lemma, one deduces that 
$$
\varrho_{q(\cdotp)}(v)=\int_{\mathbb{R}^N}|v(x)|^{q(x)}\,dx=0,
$$
which implies that $v\equiv 0$. Then, by \eqref{weaksystem2P}, one obtains $X=0$, which yields  $u\equiv 0$. Therefore, $(u,v)\equiv (0,0)$.
Hence, we established that under condition \eqref{condsyst}, the only weak solution to   \eqref{2} is $(u,v)\equiv (0,0)$. This completes the proof of Theorem \ref{theorem2}.

\end{proof}

\section*{Acknowledgments}
The third author extends his appreciation to the Deanship of Scientific Research at King Saud University, Saudi Arabia, for funding this work through research group no. RGP-237.

\end{document}